\documentclass[11pt,a4paper,reqno]{amsart}
\usepackage{fancyhdr}
\usepackage[T1]{fontenc}
\usepackage{mathrsfs}
\usepackage{amsfonts,amssymb,amsmath,amsgen,amsthm}
\usepackage{graphicx,color,epsfig}
\usepackage{epstopdf}
\usepackage{amsbsy}
\usepackage{mathrsfs}
\usepackage{bbm}
\usepackage[hidelinks]{hyperref}
\usepackage{titletoc}
\usepackage{enumerate}
\usepackage{subcaption}
\usepackage[all]{xy}
\usepackage{float}
\usepackage{mathtools}
\theoremstyle{plain}
\newtheorem{theorem}{Theorem}[section]
\newtheorem{definition}[theorem]{Definition}
\newtheorem{assumption}[theorem]{Assumption}
\newtheorem{lemma}[theorem]{Lemma}

\newtheorem{remark}[theorem]{Remark}

\newtheorem*{notations}{Notations}

\numberwithin{equation}{section}

\catcode`@=11
\def\section{\@startsection{section}{1}%
	\z@{1.5\linespacing\@plus\linespacing}{.5\linespacing}%
	{\normalfont\bfseries\large\centering}}
\catcode`@=12
\def\R{{\mathbb R}}
\def\N{{\mathbb N}}
\def\Z{{\mathbb Z}}
\def\T{{\mathbb T}}

\def\eps{\varepsilon}

\def\calL{{\mathcal L}}

\def\IE{\mathbb{E}}

\def \IR{\mathbb R}
\newcommand{\EE}{\IR^d}
\def \N{\mathbb N}

\def \eps{\epsilon}
\def\uX{\underline{X}}
\def\ux{\underline{x}}
\newcommand{\avg}[1]{\left< #1 \right>}

\title{
Uniformly accurate schemes for drift--oscillatory stochastic differential equations
}

\author{ 
Ibrahim Almuslimani\textsuperscript{\textit{\lowercase{a}},$\ast$}, Philippe Chartier\textsuperscript{\textit{\lowercase{b}}}, Mohammed Lemou\textsuperscript{\textit{\lowercase{c}}} and Florian M\'ehats\textsuperscript{\textit{\lowercase{d}}}
}

\begin{document}
\renewcommand{\thefootnote}{\textit{\alph{footnote}}}
\footnotetext[1]{Univ Rennes, INRIA, IRMAR - UMR 6625, F-35000 Rennes, France. Ibrahim.Almuslimani@irisa.fr}
\footnotetext[2]{
Univ Rennes, INRIA, CNRS, IRMAR - UMR 6625, F-35000 Rennes, France.
Philippe.Chartier@inria.fr}
\footnotetext[3]{
Univ Rennes, CNRS, IRMAR - UMR 6625, F-35000 Rennes, France.
Mohammed.Lemou@univ-rennes1.fr}
\footnotetext[4]{
Univ Rennes, CNRS, IRMAR - UMR 6625, F-35000 Rennes, France.
Florian.Mehats@univ-rennes1.fr}
\renewcommand{\thefootnote}{\fnsymbol{footnote}}
\footnotetext[1]{The work of I. Almuslimani is supported by the Swiss National Science Foundation, project No: P2GEP2\_195212.}
\renewcommand{\thefootnote}{\arabic{footnote}}
\maketitle

\begin{abstract}
In this work, we adapt the  {\em micro-macro} methodology to stochastic differential equations for the purpose of numerically solving oscillatory evolution equations. The models we consider are addressed in a wide spectrum of regimes where oscillations may be  slow or  fast. We show that through an ad-hoc transformation (the micro-macro decomposition), it is possible to retain the usual orders of convergence of  Euler-Maruyama method, that is to say, uniform weak order one and uniform strong order one half.  {We also show that the same orders of uniform accuracy can be achieved by a simple integral scheme. The advantage of the micro-macro scheme is that, in contrast to the integral scheme, it can be generalized to higher order methods.}

\smallskip

 \end{abstract}
 
 \section{Introduction}
 In this paper, we aim at  constructing  {\em uniformly accurate} numerical schemes for solving It\^o stochastic differential equations (SDEs) with a (possibly highly) oscillatory drift term of the form
 \begin{align} \label{eq:SDE}
 	d X (t)=& f_{t/\eps}(X(t)) dt + \sigma(X(t)) dW(t), \quad X(0) = X_0 \in \R^d,
 \end{align}
 where $X(t)$ is a stochastic process with values in $\R^d$.  {A standard assumption in the literature of oscillatory problems and averaging theory is that the drift function $(\theta,x)\in \T\times\R^d\mapsto f_\theta(x)$ is assumed to be periodic\footnote{ {Typically, the deterministic part of \eqref{eq:SDE} is a result of a change of variable applied to the "non-filtered" equation of the form $\dot v = -\eps^{-1}Av+g(v)$, where the matrix $A$ is such that $\exp(A)=I$, which leads to the periodic vector field $f_\theta$ \cite[Remark 1.1]{CLMV20}.}} with respect to $\theta$ (we shall denote accordingly the torus by  $\T=\R/\Z\equiv[0,1]$), see e.g. \cite{SVM07,CLMV20,CCMM15}. For simplicity and ease of the presentation, and without loss of generality, we assume in the rest of the paper that $(\theta,x)\in \T\times\R^d\mapsto f_\theta(x)$ is 1-periodic}. The diffusion function $\sigma$ is  defined as a smooth function $x\in\R^d \mapsto \sigma(x)\in\R^{d\times m}$. Finally, $W(t)=(W_1(t),...,W_m(t))^T$ is an array of $m$ independent one-dimensional Weiner processes. Precise regularity assumptions on $f$ and $\sigma$ are made in Assumption \ref{eq:hypo}. Let us emphasize that $\eps$ is here a parameter whose value can freely vary in the interval $]0,1]$ and that equation (\ref{eq:SDE}) is not restricted to its asymptotic regime where $\eps$ tends to zero {, i.e., the oscillations may be slow or fast without any restriction on their frequency}.  {This model may appear in many interesting applications including the perturbation of deterministic oscillatory  problems by adding a random noise term for modeling purposes. For example, in Section \ref{sec:num}, we consider a stochastic H\'enon-Heiles model, one might also think of stochastic Fermi-Pasta-Ulam problem and the stochastic nonlinear Schr\"odinger equation \cite{V14}.}
 
 Analogously to the case of deterministic differential equations, in the highly-oscillatory (stiff) regime ($\epsilon \ll 1)$, standard  numerical methods for SDEs, such as Euler-Maruyama method,  face a severe time step restriction ($\Delta t=\mathcal{O}(\eps)$) when applied directly to \eqref{eq:SDE}. This issue is well documented in the literature for deterministic ODEs  \cite{SVM07}, and several classes of methods were introduced in order to deal with stiffness \cite{BD12,CCMM15,CMMV14, CLM13,FS14}. However, none of the methods introduced therein qualify as uniformly accurate methods as they do not produce numerical approximations with an accuracy and at a cost both independent of the value of $\eps \in ]0,1]$. This motivated the introduction of a new methodology based on averaging techniques and exposed in \cite{CLMV20}: there, the authors elaborate a new technique enabling standard numerical methods to retain their non-stiff order with uniform accuracy for all $\eps\in]0,1]$. In the case of stiff SDEs, as for ODEs, we can differentiate between two kinds of stiffness: stiff dissipative SDEs and highly-oscillatory SDEs. While for stiff dissipative SDEs, many interesting integrators were introduced in the last two decades with nice stability and convergence properties \cite{AAV18,AbL08,AVZ13,AbV12a},  {the numerical solution of  highly-oscillatory (drift) SDEs has not received so much attention except, for example, \cite{V14} in which the author derives weak second order multirevolution composition methods which are accurate only for very small values of $\eps$ and fail for $\eps\simeq 1$}. This contribution is, up to our knowledge, the first attempt to adapt the technique of  micro-macro decomposition to the SDE context in order to construct a \emph{uniformly accurate} integrator that works well for every $\eps\in ]0,1]$. 
 
 More precisely, we focus , in the present paper, on constructing uniformly accurate methods for highly-oscillatory SDEs in the spirit of the methodology explained in \cite{CLMV20}. We derive a micro-macro system by introducing a change of variable that leads us to treat the average decay and the fast oscillations separately. We show that applying Euler-Maruyama method to the micro-macro system gives,  {under appropriate assumptions}, an approximation of uniform weak order 1 and strong order 1/2 for any value of $\eps\in(0,1]$ and with no restriction on the time step. In more mathematical terms, we prove that,  {under Assumption \ref{eq:hypo} below}, the Euler-Maruyama scheme for solving the {\em micro-macro} system derived from \eqref{eq:SDE}, provides approximations $X_n$ on a uniform grid $\{t_n\}_{n=0}^{N}$ such that 
 \begin{align*} 
 	\forall n=0, \ldots, N, \quad |\IE(\phi(X(t_n))) - \IE(\phi(X_n)) | \leq C \,  h  \mbox{ and }  
 	\IE(|X(t_n) - X_n|^2)^{\frac12} \leq Ch^{\frac12},
 \end{align*}
 where $h=t_{n+1}-t_n$ (for simplicity we consider constant time step size) and the constant $C$ is independent of $n$, $h$ and  $\eps$. 
 We prove as well that the same result can be obtained using the integral scheme introduced in Section \ref{sec:int}.  {After Noticing the simplicity of implementation and the uniform accuracy of the integral scheme, a natural question arises: why do we consider the micro-macro decomposition? The answer is, interestingly, deriving the micro-macro method gives insight about possible generalizations to higher-order schemes inspired from deterministic averaging theory developed in \cite{CLMV20}}. \\
 
 The rest of the paper is organized as follows. In the next subsection, we fix the main notations used throughout the paper and we make assumptions. In Section \ref{sec:int} we introduce the integral scheme and we prove its uniform convergence (weak and strong) with respect to the parameter $\eps$. In Section \ref{sec:mm} we derive the micro-macro scheme, and we show its uniform convergence properties. Finally, in Section \ref{sec:num} we present some numerical experiments that illustrate the efficiency of the schemes.
 \subsection*{Notations and Assumptions}
 Noteworthy, our results are obtained under quite standard assumptions that we recall below. In particular, we shall constantly suppose that Assumption \ref{eq:hypo} below holds true in the sequel. 
 \begin{definition}
 	Let $q\in\N$, we define $C_{poly}^q$ as the set of $q$-times differentiable functions from $\R^d\to\R$ whose all derivatives up to order $q$ have at most polynomial growth.  {We define as well $C_{lin}^q$ in the same way but with linear growth instead of polynomial}.
 \end{definition}
 
 \begin{assumption} \label{eq:hypo}
 	The functions $x \mapsto\sigma_{i,j}(x)$, $i=1,\ldots,d$, $j=1, \ldots,m$ are Lipschitz functions of the set  $C^2_{poly}$.
 	The function $(\theta,x)  \mapsto f_\theta(x) \in \R^d$ is defined on $\T \times \R^d$ and it is uniformly Lipschitz continuous with respect to $x$ (i.e with Lipschitz constant independent of $\theta$). In addition, its components $(f_\theta(\cdot))_i$, $i=1, \ldots, d$, are functions of  {$C_{poly}^4\cap C_{lin}^2$}. 
 	Furthermore, $f$ is continuously differentiable with respect to $\theta$ and $x\mapsto\partial_\theta f_\theta(x)$ is  uniformly Lipschitz continuous with respect to the variable $x$.
 \end{assumption}
 
 \begin{remark}
 	The Lipschitz continuity of $f_\theta$ and $\sigma$ ensures the existence and uniqueness of the solution of \eqref{eq:SDE} in $C([0,T] ; L^2(\Omega))$. It ensures in addition a linear growth, that is to say 
 	\begin{equation*}
 		\forall (\theta,x) \in \T \times \R^d, \quad |f_\theta(x)|+|\sigma(x)|\leq K(1+|x|),
 	\end{equation*}
 	where $K$ is a constant independent of $\theta$. 
 \end{remark}
 
 \begin{notations}
 	The derivative with respect to the space variable $x$ will be denoted by a \emph{prime}. For example, $f'_\theta(x)=\partial_xf_\theta(x)$ (and gradient or Jacobien when needed), and $f''_\theta(x)=\partial^2_xf_\theta(x)$ (Hessian matrix or second derivative tensor when needed). The letter $L$ will be used throughout the paper as the maximum Lipschitz constant of the concerned functions. The constants represented by the capital letter $C$ (indexed or not) are generic constants. We denote by $|\cdot|$ the Euclidean norm of $\R^d$.
 \end{notations}
 { We will conduct all our analysis for $m=1$ in \eqref{eq:SDE}. This comes for the sake of simplicity and ease of presentation, and to avoid any confusion due to the heavy computations done in the proofs. However, we emphasize that our results apply straightforwardly to SDEs of the form \eqref{eq:SDE} with any number of Weiner processes. Hence, $\sigma$ will denote a function from $\R^d\to\R^d$ and $W(t)$ will denote a standard one-dimensional Weiner process.}

 \section{Integral scheme} \label{sec:int}
 Let $T>0$ be a given final time, $N\in \N\backslash\{0\}$, and $h=T/N$, and let $t_n=nh,~n=0,\dots,N$. We consider the integral scheme
 \begin{equation}\label{eq:int}
 	X_{n+1}=X_n+h\avg{f}(X_n)+\eps\int_{t_n/\eps}^{t_{n+1}/\eps}(f_{\theta}(X_n)-\avg{f}(X_n))d\theta + \sigma(X_n)\Delta W_n,
 \end{equation}
 where $\avg{f}(x)$ is the value at $x$ of the average of $f$ defined as 
 $$
 \avg{f}(x)  = \int_\T f_\theta(x) d\theta,
 $$
 and $\Delta W_n=W(t_{n+1})-W(t_n)\sim \mathcal{N}(0,h)$ which can be replaced by $\sqrt{h}\xi_n$ where $\xi_n\sim \mathcal{N}(0,1)$.
 Note that the integral term in \eqref{eq:int} can be rewritten as $\eps(F_{t_{n+1}/\eps}(X_n)-F_{t_n/\eps}(X_n))$, where 
 $$
 F_\theta(x)  = \int_0^\theta \left( f_\tau(x) - \avg{f}(x) \right) d\tau.
 $$
 
 \begin{theorem}\label{th:convint}
 	Consider the integral scheme \eqref{eq:int}, where $t_n=nh$ for $n=0, \ldots, N$, and $N\in\N\backslash\{0\}$. 
 	Let $X(t)$ be the solution of \eqref{eq:SDE}. Then, for all $\phi \in C_{poly}^4$, there exists 
 	$C>0$ independent of $h$, $n$ and $\eps \in (0,1]$ such that, 
 	\begin{align}
 		\forall n \in \{0, \ldots, N\}, \quad &|\IE(\phi(X(t_n))) - \IE(\phi(X_n)) | \leq C  h,\\
 		\forall n \in \{0, \ldots, N\}, \quad  &\IE(|X(t_n) - X_n|^2)^{\frac12} \leq Ch^{\frac12}. \label{eq:strong}
 	\end{align}
 \end{theorem}

 \begin{lemma}\label{lemma:df}
 	Under the assumptions of Theorem \ref{th:convint}, one has for all $n=0,\dots,N-1$
 	\begin{equation} \label{eq:mean}
 		\left|\IE\left(\int_{t_n}^{t_{n+1}}[f_{t/\eps}(X(t))-f_{t/\eps}(X(t_n))]dt\right)\right|\leq Ch^2.
 	\end{equation}
 \end{lemma}
 
 \begin{proof} We first rewrite the integrand $f_{t/\eps}(X(t))-f_{t/\eps}(X(t_n))$ of (\ref{eq:mean}) as
 	\begin{align*}
 		\int_{t_n}^t d(f_{t/\eps}(X(s)))
 		&=\int_{t_n}^t [f'_{t/\eps}(X(s))f_{t/\eps}(X(s))\\
 		&+\frac12f''_{t/\eps}(X(s))(\sigma(X(s)),\sigma(X(s)))]ds\\
 		&+\int_{t_n}^tf'_{t/\eps}(X(s))\sigma(X(s))dW(s),
 	\end{align*}
 	so that 
 	\begin{align*}
 		\left|\IE\left(\int_{t_n}^{t_{n+1}}[f_{t/\eps}(X(t))-f_{t/\eps}(X(t_n))]dt\right)\right|&\leq\left|\IE\left(\int_{t_n}^{t_{n+1}}\int_{t_n}^t \biggl[f'_{t/\eps}(X(s))f_{t/\eps}(X(s))\right.\right.\\
 		&\hspace*{-6ex}\left.\left.+\frac12f''_{t/\eps}(X(s))(\sigma(X(s)),\sigma(X(s)))\biggr]dsdt\right)\right|\\
 		&\hspace*{-5ex}+\left|\IE\left(\int_{t_n}^{t_{n+1}}\int_{t_n}^tf'_{t/\eps}(X(s))\sigma(X(s))dW(s)dt\right)\right|,
 	\end{align*}
 	and using Fubini's theorem, we get the following upper-bound of the left-hand side of (\ref{eq:mean})
 	\begin{align*}
 		\left|\int_{t_n}^{t_{n+1}}\IE\left(\int_{t_n}^tf'_{t/\eps}(X(s))\sigma(X(s))dW(s)\right)dt\right| + \mathcal{O}(h^2),
 	\end{align*}
 	where $\mathcal{O}(h^2)$ is bounded thanks to the polynomial growth assumptions and the boundedness of the moments. The expectation in previous integral is an expectation of a stochastic integral, hence it is null since $X(s)$ is independent of the increments of $W$ for times above $s$.
 \end{proof}
 
 %
 
 \begin{lemma}\label{lemma:strong}
 	Under the assumptions of Theorem \ref{th:convint}, we have for all $n=0,\dots,N-1$
 	\begin{equation*}
 		\IE(|X(t)-X(t_n)|^2~;~X(t_n)=x)\leq C(x)(t-t_n).
 	\end{equation*}
 \end{lemma}
 
 \begin{proof} From the integral form of equation (\ref{eq:SDE}) 
 	\begin{equation*}
 		X(t)=X(t_n)+\int_{t_n}^{t}f_{s/\eps}(X(s)ds+\int_{t_n}^t\sigma(X(s))dW(s).
 	\end{equation*}
 	we may write 
 	\begin{align*}
 		\IE(|X(t)-X(t_n)|^2~;~X(t_n)=x)&= \IE\left(\left|\int_{t_n}^tf_{s/\eps}(X(s))ds + \int_{t_n}^t\sigma(X(s))dW(s)\right|^2\right)\\
 		&\hspace*{-15ex}\leq 2(t-t_n)\int_{t_n}^t\IE(f_{s/\eps}(X(s))^2)ds +2\int_{t_n}^t\IE(\sigma(X(s))^2)ds\\
 		&\hspace*{-15ex}\leq 2(t-t_n)\int_{t_n}^tK(1+\IE(|X(t_n)|^2))ds+2\int_{t_n}^tK(1+\IE(|X(t_n)|^2))ds\\
 		&\hspace*{-15ex}\leq C(x)(t-t_n),
 	\end{align*}
 	where we have used, between the first and second lines, the Young's inequality, then the Cauchy-Schwarz inequality for the first integral, and It\^o isometry for the second integral. Finally, we have used the bound $\IE(|X(s)|^2)\leq K(1+|x|^2)$,  borrowed from \cite{GS72}, between  the third and the fourth lines.
 \end{proof}

 \begin{proof}[Proof of  Theorem \ref{th:convint}] We hereby follow the methodology introduced in  \cite{talay86}, which consists in bounding the moments (first step), then proving the weak convergence of the scheme (second step) and finally  establishing its strong convergence (third step).
 	\subsection*{Step $1$} In order to bound the moments of $X_n$ of arbitrary order, we shall resort to \cite[Lemma 2.2, p. 102]{Mil04}, which requires the following estimates
 	\begin{align*}
 		|\IE(X_{n+1}-X_n|X_n=x)|&=\left|\IE\left(\eps\int_{t_n/\eps}^{t_{n+1}/\eps}f_{\theta}(x)d\theta\right)\right|\\
 		&\leq\eps\int_{t_n/\eps}^{t_{n+1}/\eps}K(1+|x|)d\theta
 		=K(1+|x|)h,
 	\end{align*}
 	and 
 	\begin{align*}
 		|X_{n+1}-X_n|&=\left|\eps\int_{t_n/\eps}^{t_{n+1}/\eps}f_{\theta}(X_n)d\theta + \sigma(X_n)\Delta W_n\right|\\
 		&\leq K(1+|X_n|)h+K(1+|X_n|)\frac{|\Delta W_n|}{\sqrt{h}}\sqrt{h}
 		\leq  M_n(1+|X_n|)\sqrt{h},
 	\end{align*}
 	where $M_n=K \left(\sqrt{h}+\frac{|\Delta W_n|}{\sqrt{h}}\right)$ which is of bounded moments since $\frac{\Delta W_n}{\sqrt{h}} \sim \mathcal{N}(0,1)$.
 	
 	\subsection*{Step $2$}
 	
 	Let $\phi\in C^4_{poly}$ be a test function. For $X(t_n)=X_n=x$, by performing Taylor expansion of $\phi$ with integral remainder, it can be shown that 
 	\begin{align*}
 		&|\IE(\phi(X(t_{n+1}))-\phi(X_{n+1}))| \leq \biggl|\IE\left[\phi'(x)\left(\int_{t_n}^{t_{n+1}}(f_{t/\eps}(X(t))-f_{t/\eps}(X(t_n)))dt\right)\right]\biggr|\\
 		&+\biggl|\IE\left[\phi''(x)\left(\int_{t_n}^{t_{n+1}}(\sigma(X(t))-\sigma(X(t_n)))dW(t)\right)\left(\int_{t_n}^{t_{n+1}}(\sigma(X(t))+\sigma(X(t_n)))dW(t)\right)\right]\biggr|\\
 		&+\frac16\biggl|\IE\left[\int_0^1(1-\tau)^3\phi^{(4)}(x+\tau(X(t_{n+1})-x))d\tau\,(X(t_{n+1})-x)^4\right]\biggr|\\
 		&+\frac16\biggl|\IE\left[\int_0^1(1-\tau)^3\phi^{(4)}(x+\tau(X_{n+1}-x))d\tau\,(X_{n+1}-x)^4\right]\biggr|+Ch^2\\
 		&\coloneqq \text{I+II+III+IV}+Ch^2,
 	\end{align*}
 	where the term $Ch^2$ comes from the remaining expectations of the second and third derivatives of $\phi$ applied to the integrals, which are zero for odd number of stochastic integrals, and bounded by $Ch^2$ otherwise.  {The constant C is independent of $\epsilon$ thanks to the boundedness of the moments of the exact and numerical solutions, and to the regularity assumptions made in Assumption \ref{eq:hypo}}. By Lemma \ref{lemma:df}, Lemma \ref{lemma:strong} and the boundedness of the moments, the terms I, III and IV are $\mathcal{O}(h^2)$.
 	As for the integrand of the second term, we have using It\^o formula
 	\begin{align*}
 		\sigma(X(t))-\sigma(X(t_n)) 
 		&=\int_{t_n}^t\biggl(\sigma'(X(s))f_{s/\eps}(X(s))\bigg.\\
 		&\bigg.+\frac12\sigma''(X(s))(\sigma(X(s)),\sigma(X(s)))\biggr)ds\\
 		&+\int_{t_n}^t\sigma'(X(s))\sigma(X(s))dW(s).
 	\end{align*}
 	Hence, using the Lipschitz-continuity  of $\sigma$, the polynomial growth of $\phi''$, Lemma \ref{lemma:strong}, and the following consequence of It\^o isometry 
 	\begin{equation*}
 		\forall g, h \in L^2_{ad}([a,b]\times\Omega), ~ \IE\left[\left(\int_a^bg(t)dW(t)\right)\left(\int_a^bh(t)dW(t)\right)\right]=\IE\left[\int_a^bg(t)h(t)dt\right]~~ 
 	\end{equation*}
 	we have
 	\begin{align*}
 		\text{II} &\leq C(x) \biggl|\IE\left[\left(\int_{t_n}^{t_{n+1}}(\sigma(X(t))-\sigma(X(t_n)))dW(t)\right)\left(\int_{t_n}^{t_{n+1}}(\sigma(X(t))-\sigma(X(t_n)))dW(t)\right.\right.\\
 		&\left.\left.+2\int_{t_n}^{t_{n+1}}\sigma(X(t_n))dW(t)\right)\right]\biggr|\\
 		&=C(x)\biggl|\IE\left[\left(\int_{t_n}^{t_{n+1}}(\sigma(X(t))-\sigma(X(t_n)))^2dt\right)\right]\biggr|\\
 		&+2C(x)\biggl|\IE\left[\left(\int_{t_n}^{t_{n+1}}\sigma(X(t_n))(\sigma(X(t))-\sigma(X(t_n)))dt\right)\right]\biggr|\\
 		&\leq C(x)L\int_{t_n}^{t_{n+1}}C(t-t_n)dt+2\left|\IE\left[\sigma(X(t_n))\int_{t_n}^{t_{n+1}}\int_{t_n}^td\sigma(X(s))dt\right]\right|\\
 		&=2\left|\IE\left[\sigma(X(t_n))\int_{t_n}^{t_{n+1}}\int_{t_n}^t\biggl(\sigma'(X_s)f_{s/\eps}(X_s)+\sigma''(X_s)(\sigma(X_s) ,\sigma(X_s) )\biggr)dsdt\right]\right|\\
 		&+2\left|\IE\left[\sigma(X(t_n))\int_{t_n}^{t_{n+1}}\int_{t_n}^t\sigma'(X_s)\sigma(X_s)dW(s)dt\right]\right|+C_1h^2.
 	\end{align*}
 	The first expectation after the last equal sign is clearly $\mathcal{O}(h^2)$. The second one is equal to zero since $X(t_n)$ is independent of the increments $W(r)-W(s)$ for $t_n\leq s \leq r$. The boundedness of the moments and the local weak order $2$ imply the global weak convergence of order  $1$ by a theorem from \cite{Mil95} (see also \cite[Chapter 2.2]{Mil04}).

 	
 	\subsection*{Step 3}
 	We first derive an upper-bound of $\IE(|X(t_{n+1})-X_{n+1}|^2; X(t_n)=x)$ as follows
 	\begin{align*}
 		&
 		\IE\biggl(\biggl|\int_{t_n}^{t_{n+1}}(f_{t/\eps}(X(t))-f_{t/\eps}(x))dt+\int_{t_n}^{t_{n+1}}(\sigma(X(t))-\sigma(x))dW(t)\biggr|^2\biggr)\\
 		&\leq  2\IE\biggl(\biggl|\int_{t_n}^{t_{n+1}}(f_{t/\eps}(X(t))-f_{t/\eps}(x))dt\biggr|^2\biggr) +
 		2\IE\biggl(\biggl|\int_{t_n}^{t_{n+1}}(\sigma(X(t))-\sigma(x))dW(t)\biggr|^2\biggr)\\
 		&\leq 2h\IE\biggl(\int_{t_n}^{t_{n+1}}\biggl|f_{t/\eps}(X(t))-f_{t/\eps}(X(t_n))\biggr|^2dt\biggr) +2\IE\biggl(\int_{t_n}^{t_{n+1}}\biggl|\sigma(X(t))-\sigma(X(t_n))\biggr|^2dt\biggr)\\
 		&\leq 2L^2(h+1) \int_{t_n}^{t_{n+1}}\IE(|X(t)-X(t_n)|^2)dt \leq 2L^2(h+1)\int_{t_n}^{t_{n+1}}C(t-t_n)dt \leq Ch^2,
 	\end{align*}
 	where we have used Lemma \eqref{lemma:strong}. Finally, we conclude that $\IE(|X(t_{n+1})-X_{n+1}|^2)^{\frac12}\leq Ch$. A well-known theorem by Milstein \cite{Mil87} then allows to establish inequality (\ref{eq:strong}). 
 \end{proof}
 
 { Remark that for the integral scheme we can relax the assumption $(f_\theta(.))_i,\,i=1,\dots,d\in C_{poly}^4\cap C_{lin}^2$ to $(f_\theta(.))_i,\,i=1,\dots,d\in C_{poly}^2$.}
 \section{Micro-Macro method}\label{sec:mm}
 { The integral scheme, despite its uniform accuracy, does not generalize to higher order methods. For this reason, and inspired by \cite{CCMM15,CLMV20}, we will introduce and analyze a micro-macro separation of scales that leads the standard Euler-Maruyama method to retain its usual weak and strong orders of convergence. We believe that this will be an excellent starting point and a good insight to try to develop higher order micro-macro methods in future works.}
 
  {As in general averaging theory, the purpose is to to find a periodic, near-identity and smooth change of variable $\Phi_\theta$, together with a flow $\Psi_t$, the flow map of an autonomous non-stiff differential equation on $\EE$, such that the solution of the original equation \eqref{eq:SDE} takes the composed form \cite{SVM07,LM88,CCMM15,CLMV20}
 	\begin{equation}\label{eq:decomp}
 		X(t)=\Phi_{t/\eps}\circ\Psi_t(X_0).
 \end{equation}}
  {In the current work, we restrict ourselves to first order averaging, and we show that the first order change of variable $\Phi_\theta$ defined, for all $0 < \eps \leq 1$, by the formula}
 \begin{align}\label{eq:Phi}
 	\Phi_\theta(x) = x + \eps F_\theta(x)  = x + \eps \int_0^\theta \left( f_\tau(x) - \avg{f}(x) \right) d\tau,
 \end{align}
  {and the first order flow $\Psi_t$ satisfying the autonomous equation
 	\begin{equation}\label{eq:Psi}
 		d\Psi_t(x)=\avg{f}(\Psi_t(x))dt+\sigma(\Psi_t(x))dW(t),\qquad \Psi_0(x)=x,
 	\end{equation}
 	both borrowed from deterministic averaging, work well to construct a \emph{uniform accurate micro-macro scheme of weak order 1 and strong order 1/2} for \eqref{eq:SDE}. For the derivation of \eqref{eq:Phi} and \eqref{eq:Psi} from \eqref{eq:decomp}, we refer to \cite{CLMV20}.}
 
  {Let us denote by $\uX$ the solution of \eqref{eq:Psi}}, separating slow and fast scales as follows
 \begin{align} \label{eq:dec}
 	X = \Phi_{t/\eps}( \uX ) + Y,
 \end{align}
 leads,  {using It\^o formula}, to the micro-macro system  of the form 
 \begin{align} 
 	d\uX &= \avg{f}(\uX) dt + \sigma(\uX) dW,  \quad \uX(0) = X_0, \label{eq:mmXbar}\\
 	dY &= \left(f_{t/\eps}\left( \Phi_{t/\eps}( \uX ) + Y\right)-f_{t/\eps}( \uX) 
 	- \eps F_{t/\eps}'(\uX) \avg{f}(\uX)\right. \label{eq:mmY}\\
 	&\left.- \frac{\eps}{2} F_{t/\eps}''(\uX) \left(\sigma(\uX) ,\sigma(\uX) \right) \right) dt \nonumber  \\
 	& + \left(\sigma(\Phi_{t/\eps}( \uX ) + Y) - \sigma(\uX) - \eps F_{t/\eps}'(\uX) \sigma(\uX)\right) dW, \quad Y(0) = 0. \nonumber 
 \end{align}
  {Here $Y$ represents the averaging error that is characterized by the equation \eqref{eq:mmY} and added after applying the change of variable in order to recover the solution $X(t)$ of \eqref{eq:SDE}. We will show later that the expected value of $Y(t)$ is of size $\eps$ (see Lemma \ref{lemma:est}), which is analogous to the averaging error for deterministic oscillatory equations ($\sigma\equiv0$).}
 
 { \begin{remark}
 		The existence and the uniqueness of the solution $Y(t)$ of \eqref{eq:mmY} is an immediate consequence of the existence and the uniqueness of the solutions $X(t)$ and $\uX(t)$ of \eqref{eq:SDE} and \eqref{eq:mmXbar}, and the decomposition \eqref{eq:dec}.
 \end{remark}}
 
 We use the same uniform discretization as for the integral scheme. Our aim is now to prove a uniform (in $\eps$) convergence result for the following micro-macro scheme, which is nothing but the  Euler-Maruyama method applied to (\ref{eq:mmXbar}, \ref{eq:mmY})
 \begin{align}
 	\uX_{n+1} &= \uX_n + h \avg{f} (\uX_n) + \sqrt{h} \sigma(\uX_n) \xi_n, \quad \uX_0 = X_0, \label{eq:EMXbar}\\
 	Y_{n+1} &= Y_n + h \left(f_{t/\eps}\left( \Phi_{t_n/\eps}( \uX_n ) + Y\right)-f_{t_n/\eps}( \uX_n) \right)  \label{eq:EMY} \\
 	& - \eps h \left(F_{t_n/\eps}'(\uX_n) \avg{f}(\uX_n) + \frac{1}{2}  F_{t_n/\eps}''(\uX_n) \left(\sigma(\uX_n) ,\sigma(\uX_n) \right) \right)  \nonumber \\
 	&+ \sqrt{h} \left(\sigma(\Phi_{t_n/\eps}( \uX_n ) + Y_n) - \sigma(\uX_n) - \eps F_{t_n/\eps}'(\uX_n) \sigma(\uX_n)\right)  \xi_n, \quad Y_0 = 0,\nonumber 
 \end{align}
 where the increment $\xi_n$ is a random quantity sampled from a normalized Gaussian centered at zero and with variance $1$. 
 \subsection{Main result}
 \begin{theorem}\label{th:main}
 	Consider the Euler-Maruyama scheme (\ref{eq:EMXbar}, \ref{eq:EMY}) for solving the {\em micro-macro} system (\ref{eq:mmXbar}, \ref{eq:mmY}) and let 
 	$$
 	X_n = \Phi_{t_n/\eps}(\uX_n) + Y_n
 	$$
 	for $n=0,\dots,N$.
 	Let $X(t)$ be the solution of \eqref{eq:SDE}. Then, for all $\phi \in C_{poly}^4$, there exists 
 	$C>0$ independent of $h$, $n$, and $\eps$, such that, one has
 	\begin{align} \label{eq:mainest}
 		\forall n \in \{0, \ldots, N\}, \quad &|\IE(\phi(X(t_n))) - \IE(\phi(X_n)) | \leq C \,  h,\\
 		\forall n \in \{0, \ldots, N\}, \quad &\IE(|X(t_n) - X_n|^2)^{\frac12} \leq Ch^{\frac12}. 
 	\end{align}
 \end{theorem}
 Again, the proof of the theorem follows the usual steps from \cite{talay86}. Note that the main novelty of our result lies in the fact that estimate \eqref{eq:mainest} is uniform w.r.t. $\eps \in (0,1]$.  The following lemma will be needed in the proof of Theorem \ref{th:main}.
 \begin{lemma}\label{lemma:est}
 	There exists $c>0$ such that for all $\eps \in (0,1]$ and all $t\in[0,T]$, one has 
 	$$\IE(|Y(t)|)\leq c\eps.
 	$$
 \end{lemma}
 
 \begin{proof}
 	First, note that we have for all $t\in[0,T]$ and all positive integer $m$  (see \cite{GS72}) 
 	$$\IE(|\uX(t)|^{2m})\leq K(1+|X_0|^{2m}).$$ 
 	Moreover, we have 
 	\begin{align*}
 		&\IE(|Y(t)|^2)\leq 2t\IE\left(
 		\int_0^t \biggl|f_{s/\eps}(\Phi_{s/\eps}(\uX(s))+Y(s)) - f_{s/\eps}(\uX(s))\right.\\
 		&\left.- \eps  \biggl( F'_{s/\eps}(\uX(s))\avg f(\uX(s))+\frac12  F''_{s/\eps}(\uX(s))(\sigma(\uX(s)) ,\sigma(\uX(s)) )\biggr)\biggr|^2
 		ds \right)\\
 		& + 2\IE\left(
 		\int_0^t \biggl|
 		\sigma(\Phi_{s/\eps}(\uX(s))+Y(s))-\sigma(\uX(s))-\eps F'_{s/\eps}(\uX(s))\sigma(\uX(s))
 		\biggr|^2
 		ds \right)\\
 		&\leq 4t
 		\int_0^t \IE\left(\biggl|f_{s/\eps}(\Phi_{s/\eps}(\uX(s))+Y(s)) - f_{s/\eps}(\uX(s))\biggr|^2\right)ds\\
 		&+4t\eps^2  \int_0^t\IE\left(\biggl| F'_{s/\eps}(\uX(s))\avg f(\uX(s))+\frac12  F''_{s/\eps}(\uX(s))(\sigma(\uX(s)) ,\sigma(\uX(s)) )\biggr|^2
 		\right)ds \\
 		& + 4
 		\int_0^t \IE\left(\biggl|
 		\sigma(\Phi_{s/\eps}(\uX(s))+Y(s))-\sigma(\uX(s))\biggr|^2 \right)ds
 		+4\eps^2 \int_0^t \IE\left(\biggl|F'_{s/\eps}(\uX(s))\sigma(\uX(s))
 		\biggr|^2
 		\right)ds \\
 		&\coloneqq 4tA+4t\eps^2B+4C+4\eps^2D.
 	\end{align*}
 	
 	In the first inequality, we have used the Cauchy-Schwarz inequality for the first term and It\^o isometry for the second one. We have also used the triangular inequality as well as Young's inequality.
 	Given that $\IE(|\uX(s)|^{2m})< +\infty$ and that the components of $F_\theta$  are in $C^2_{poly}$ by Assumption \ref{eq:hypo}, we can conclude that the terms $B$ and $D$ are uniformly bounded. Now, by definition of the change of variables \eqref{eq:Phi} and the uniform Lipschitz continuity of $f_\theta$ and $\sigma$,  we get
 	\begin{align*}
 		4tA+4C& \leq 4tL^2 \int_0^t \IE(|\eps F'_{s/\eps}(\uX(s))+Y(s)|^2)ds+4L^2 \int_0^t \IE(|\eps F'_{s/\eps}(\uX(s))+Y(s)|^2)ds\\
 		&\leq 8L^2(T+1)\eps^2\int_0^t\IE(| F'_{s/\eps}(\uX(s))|^2)ds
 		+8L^2(T+1)\int_0^t \IE(|Y(s)|^2)ds\\
 		&\leq C_1\eps^2 + C_2 \int_0^t \IE(|Y(s)|^2)ds.
 	\end{align*}
 	Therefore, $\IE(|Y(t)|^2)$ satisfies the following inequality
 	$$
 	\IE(|Y(t)|^2)\leq C_3\eps^2+C_4\int_0^t \IE(|Y(s)|^2)ds.
 	$$
 	Gronwall's lemma imply that $\IE(|Y(t)|^2)\leq c^2\eps^2$. Finally, $\IE(|Y(t)|)^2\leq\IE(|Y(t)|^2)=c^2\eps^2$, thus $\IE(|Y(t)|)\leq c\eps$ where $c$ is a positive constant independent of $\epsilon$ and $t$.
 	
 \end{proof}

 \begin{proof}[Proof of  Theorem \ref{th:main}]
 	We begin by the proof of the \textbf{weak convergence}: Collecting the variables $\uX$ and $Y$ into $Z = (\uX,Y)$ and similarly $z=(\ux, y)$,  we may rewrite the micro-macro system (\ref{eq:mmXbar}, \ref{eq:mmY}) as 
 	$$
 	dZ = g_{t/\eps}(Z) dt + \Sigma_{t/\eps}(Z) dW
 	$$
 	where 
 	$$
 	g_\theta(z) = \left( 
 	\begin{array}{c}
 		\avg{f}(\ux) \\
 		f_{\theta}\left( \Phi_{\theta}( \ux ) + y\right)-f_{\theta}( \ux) 
 		- \eps F_{\theta}'(\ux) \avg{f}(\ux) - \frac{\eps}{2}    F_{\theta}''(\ux) \left(\sigma(\ux)  ,\sigma(\ux)   \right)
 	\end{array}
 	\right)
 	$$
 	and 
 	$$
 	\Sigma_\theta(z) = \left( 
 	\begin{array}{c}
 		\sigma(\ux) \\
 		\sigma(\Phi_{\theta}( \ux ) + y) - \sigma(\ux) - \eps F_{\theta}'(\ux) \sigma(\ux)
 	\end{array}
 	\right)
 	$$
 	and accordingly equations \eqref{eq:EMXbar}, \eqref{eq:EMY} as 
 	$$
 	Z_{n+1} = Z_n + h g_{t_n/\eps}(Z_n) + \sqrt{h} \Sigma_{t_n/\eps}(Z_n) \xi_n
 	$$
 	where the $\xi_n \sim \mathcal{N}(0,I_m)$ are independent Gaussian random variables.
 	\subsection*{Step $1$}
 	{Since the vector fields $g_{t/\eps}$ and $\Sigma_{t/\eps}$ have uniform linear growth in $z$ (this follows from the linear growth of $f$, $\sigma$, $\Phi$, $F$, $F'$, and $F''$)}, the proof of the boundedness of the moments of the numerical solution can be obtained  following the same arguments as in  Milstein's lemma for the non-oscillating case $(\eps=1)$. As a matter of fact, we have that 
 	\begin{itemize}
 		\item[(i)] $\IE(|Z_0|^{2m})<+\infty$ since $Z_0=(X_0,0)\in\R^{2d}$ (deterministic);
 		\item[(ii)] $|\IE(Z_{n+1}-Z_n~|~Z_n=z)|=h|g_{t_n/\eps}(z)|\leq K(1+|z|)h$;
 		\item[(iii)] $|Z_{n+1}-Z_n|\leq h|g_{t_n/\eps}(Z_n)|+|\Sigma_{t_n/\eps}(Z_n)||\Delta W_n|\leq M_n(1+|Z_n|)\sqrt{h}$, with $\Delta W_n=W(t_{n+1})-W(t_n)\sim \mathcal{N}(0,h)$ and $M_n=K|\frac{\Delta W_n}{\sqrt{h}}|$ has clearly bounded moments uniformly in $n$.
 	\end{itemize}
 	We have used the fact that $g$ and $\Sigma$ grow linearly in $Z$ thanks to Assumption \ref{eq:hypo}.
 	Under the above conditions, Lemma 2.2 from \cite[p. 102]{Mil04} implies the boundedness of the moments of arbitrary order of $Z_n$.
 	
 	\subsection*{Step $2$} 	We define, for a function $\psi$ that depends explicitly on $\theta$ and $z$,
 	\begin{equation}\label{eq:L}
 		(\calL_{\theta}\psi)(z)=\partial_{\theta}\psi(\theta,z)+\partial_{z}\psi(\theta,z)g_{\theta}(z)+\frac12 \partial^2_{z}{\psi}(\theta,z)\left(\Sigma_{\theta}(z) ,\Sigma_{\theta}(z) \right).
 	\end{equation}
 	We have for any test function $\phi\in C_{poly}^4$ (independent of theta)
 	\begin{align*}
 		\IE(\phi(Z_{n+1})|Z_n=z) &\leq \phi(z) + h (\mathcal{L}_{t_n/\eps} \phi)(z)+Ch^2 \\
 		& \hspace*{-3ex}+ \frac{1}{6} \IE\left( \left. \int_0^1 (1-\tau)^3 \phi^{(4)} (z + \tau (Z_{n+1}-z))  d \tau \;  (Z_{n+1}-z)^4   \right| Z_n=z \right). 
 	\end{align*}
 	where, as in the proof of Theorem \ref{th:convint}, the constant $C$ is independent of $\eps$ and where the term $Ch^2$ comes from the remaining expectations of the second and third derivatives of $\phi$ applied repeatedly to $hg+\sqrt{h}\Sigma\xi_n$, which are zero for odd moments of $\xi_n$, and bounded by $Ch^2$ otherwise. For the part of the remainder coming from the above first order Taylor expansion of the expectation of $\phi(Z_{n+1})$, it is clearly bounded by $Ch^2$, where the constant $C$ is independent of $\eps$. This stems from  the polynomial growth of the test function $\phi$ and its derivatives up to order $4$, as well as from the bound $|Z_{n+1}-Z_n|\leq M_n(1+|Z_n|)\sqrt{h}$.
 	
 	Performing the Taylor expansion of the expectation of the test function $\phi$ applied to the exact solution leads to
 	$$
 	\IE(\phi(Z(t_{n+1}))|Z_n=z) = \phi(z) + h (\mathcal{L}_{t_n/\eps} \phi)(z) +R,
 	$$
 	where, according to \cite[P. 26]{Mil95} \footnote{There are three other terms with vanishing expectations.}
 	\begin{equation*}
 		R=\IE\left(\int_{t_n}^{t_{n+1}}\left(\int_{t_n}^{s}(\calL_{\tau/\eps}^2\phi)(Z(\tau))d\tau\right)ds|Z(t_n)=z\right).
 	\end{equation*}
 	Let us notice that the derivative with respect to $\theta$ will not appear in $\calL_{t_n/\eps}\phi$ because our test function $\phi$ does not depend explicitly on time.  However, it will appear in $\calL_{t_n/\eps}^2\phi$ owing to the explicit dependence of $g_{t/\eps}$ and $\Sigma_{t/\eps}$ on $t$, and only because of this derivation with respect to $\theta$, we will have terms of order $\mathcal{O}(\frac1\eps)$ that need to be bounded.
 	
 	Let $g^1_{\theta}$ and $g^2_{\theta}$ be the two components of $g_\theta$ (both belong to $\R^d$), then $\partial_\theta g_\theta(z)=(\partial_\theta g_\theta^1(z),\partial_\theta g_\theta^2(z))^T=(0,\partial_\theta g_\theta^2(z))^T$, with
 	
 	\begin{align*}
 		\partial_\theta (g_\theta^2(z))=\partial_\theta (g_\theta^2(\ux,y))&=\frac1\eps\partial_\theta f_{t/\eps}(\Phi_{t/\eps}(\ux)+y)+\frac1\eps f'_{t/\eps}(\Phi_{t/\eps}(\ux)+y)\partial_\theta\Phi_{t/\eps}(\ux)\\
 		&-\frac1\eps\partial_\theta f_{t/\eps}(\ux)-\partial_\theta F'_{t/\eps}(\ux)\avg{f}(\ux)\\
 		&-\frac12  \partial_\theta F''_{t/\eps}(\ux)(\sigma(\ux) ,\sigma(\ux) ).
 	\end{align*}
 	Using the same notations for $\Sigma_\theta(z)$, we have
 	\begin{align*}
 		\partial_\theta(\Sigma_{t/\eps}^2(z))=\frac1\eps\sigma'(\Phi_{t/\eps}(\ux)+y)\partial_\theta\Phi_{t/\eps}(\ux)-\partial_\theta F'_{t/\eps}(\ux)\sigma(\ux).
 	\end{align*}
 	It remains  to bound the difference 
 	
 	\begin{equation}\label{eq:bound}
 		\frac1\eps\partial_\theta f_{\tau/\eps}(\Phi_{\tau/\eps}(\uX(\tau))+Y(\tau))-\frac1\eps\partial_\theta f_{\tau/\eps}(\uX(\tau)),
 	\end{equation}	
 	since all other terms are bounded independently of $\eps$ thanks to the smoothnes of $f_\theta$, $\sigma$ and their derivatives with respect to the space variable, and using the fact that $\partial_\theta \Phi_{\tau/\eps}(X)=\mathcal{O}(\eps)$. Now, taking into account the  uniform Lipschitz continuity of  $\partial_\theta f_\theta$, the polynomial growth of the test function $\phi$, and the above lemma, we have
 	
 	\begin{equation*}
 		\begin{split}
 			\left|\IE\left(\frac1\eps\partial_\theta f_{\tau/\eps}(\Phi_{\tau/\eps}(\uX(\tau))+Y(\tau))-\frac1\eps\partial_\theta f_{\tau/\eps}(\uX(\tau))\right)\right|&\leq \frac L\eps\IE\left|\Phi_{\tau/\eps}(\uX(\tau))+Y(\tau)-\uX(\tau)\right|\\
 			&\leq \frac L\eps\IE\left|\eps F_{\tau/\eps}(\uX(\tau))\right|+\frac L\eps\IE\left|Y(\tau)\right|\\
 			&=L\IE\left|F_{\tau/\eps}(\uX(\tau))\right|+Lc\\
 			&\leq L(K(1+\IE\left|\uX(\tau)\right|)+c)<+\infty
 		\end{split}
 	\end{equation*}
 	independently of the value of $\eps$. Hence, we have
 	\begin{equation*}
 		\left|\IE\left(\int_{t_n}^{t_{n+1}}\left(\int_{t_n}^{s}(\calL_{\tau/\eps}^2\phi)(Z(\tau))d\tau\right)ds|Z(t_n)=z\right)\right|\leq C'h^2,
 	\end{equation*}
 	and the local order 2 is proved.
 	
 	\subsection*{Step $3$} 
 	The boundedness of the moments and the local weak order $2$ imply the global weak convergence of order  $1$ by a theorem from \cite{Mil95} (see also \cite[chap2.2]{Mil04}).\\
 	
 	We have proved, for any test function $\phi\in C^4_{poly}(\R^{2d},\R)$, for all $n=0,1,\dots,N$, that 
 	\begin{equation*}
 		\left|\IE(\phi(Z(t_n))-\phi(Z_n))\right|\leq Ch,
 	\end{equation*}
 	where $Z=(\uX,Y)$ and $C$ is independent of $n$ and $\eps$. We need to prove that for any test function $\phi\in C^4_{poly}(\R^d,\R)$, for all $n=0,1,\dots,N$,
 	\begin{equation*}
 		\left|\IE(\phi(X(t_n))-\phi(X_n))\right|=\left|\IE(\phi(\Phi_{t_n/\eps}(\uX(t_n))+Y(t_n))-\phi(\Phi_{t_n/\eps}(\uX_n)+Y_n))\right|\leq C_1h.
 	\end{equation*}
 	where $C_1$ is independent of $n$ and $\eps$. Let $\phi\in C^4_{poly}(\R^d,\R)$, for each fixed parameter $\theta \in \mathbb{T}$, we consider the test function
 	\begin{equation*}
 		\psi_\theta(Z) = \phi(\Phi_\theta(\uX)+Y).
 	\end{equation*}
 	Note that $\psi_\theta \in C^4_{poly}(\R^d,\R)$, since $\Phi_\theta\in C^4_{poly}(\R^d,\R^d)$ (by assumption on $f_\theta$). Hence, for each $n=0,1,\dots,N$, we have
 	\begin{equation*}
 		\begin{split}
 			\left|\IE(\psi_{t_n/\eps}(Z(t_n))-\psi_{t_n/\eps}(Z_n))\right|&=\left|\IE(\phi(\Phi_{t_n/\eps}(\uX(t_n))+Y(t_n))-\phi(\Phi_{t_n/\eps}(\uX_n)+Y_n))\right|\\	&=\left|\IE(\phi(X(t_n))-\phi(X_n))\right|\leq C_1h,
 		\end{split}
 	\end{equation*}
 	where $C_1=C$. \\
 	
 	\paragraph{\textbf {Strong convergence}} For the strong convergence with order 1/2 (for $Z$), in addition to the bounded moments and the local weak order $2$, using \cite{Mil87}, it is sufficient to show first order  strong convergence after one step, i.e, 
 	\begin{equation}\label{eq:locstr}
 		\IE(|Z(t_{n+1})-Z_{n+1}|^2 ; Z(t_n)=z)^{\frac12}\leq Ch,
 	\end{equation}
 	where the generic constant $C$ is independent of $n, h$, and $\eps$.  By Wagner-Platen expansion  \cite[chap1.2.2]{Mil95}, we have
 	
 	\begin{align*}
 		Z(t_{n+1})-Z_{n+1}&=\hspace*{-0.5ex}\int_{t_n}^{t_{n+1}}\hspace*{-1.5ex}\int_{t_n}^{s}\Lambda_{\tau/\eps} \Sigma_{\tau/\eps}(Z(\tau))dW(\tau)dW(s)+ \int_{t_n}^{t_{n+1}}\hspace*{-1.5ex}\int_{t_n}^{s}\calL_{\tau/\eps} \Sigma_{\tau/\eps}(Z(\tau))d\tau dW(s)\\
 		& + \int_{t_n}^{t_{n+1}}\hspace*{-1.5ex}\int_{t_n}^{s}\Lambda_{\tau/\eps} g_{\tau/\eps}(Z(\tau))dW(\tau) d(s)+ \int_{t_n}^{t_{n+1}}\hspace*{-1.5ex}\int_{t_n}^{s}\calL_{\tau/\eps} g_{\tau/\eps}(Z(\tau))d\tau d(s)
 	\end{align*}
 	where $\Lambda \varphi_\theta (z)= \varphi'_\theta (z)\Sigma_\theta(z)$, $\calL_\theta$ is defined in \eqref{eq:L} and 
 	\begin{equation*}
 		(\Lambda_\theta\psi)(z)=\psi'(z)\Sigma_\theta(z).
 	\end{equation*}
 	The only terms we need to bound are the derivatives with respect to $\theta$ of $g_\theta(Z)$ and $\Sigma_\theta(Z)$ at $\theta=\tau/\eps$ arising from the application of the differential operator $\calL_{\tau/\eps}$ to $g_{\tau/\eps}(Z)$ and $\Sigma_{\tau/\eps}(Z)$ respectively.  It can be checked that the other terms in the integrals are again uniformly bounded by the regularity assumptions, Lemma \ref{lemma:est}, and the boundedness of the moments of $Z(t)$. Once this is done, we can conclude that \eqref{eq:locstr} is satisfied. The uniform boundedness of $\partial_{\theta} g_{\tau/\eps}(Z)$ was already proved in the weak convergence case. Now, we have  $\partial_\theta \Sigma_{\tau/\eps}(Z)=(0,\partial_\theta\Sigma_\theta^2(Z))$, and
 	$$
 	\partial_\theta\Sigma^2_{\tau/\eps}(Z) = 
 	\frac1\eps\sigma'(\Phi_{\tau/\eps}( \uX ) +  Y) \partial_\theta\Phi_{\tau/\eps}(\uX) - \partial_\theta F_{\tau/\eps}'(\uX) \sigma(\uX)
 	$$
 	
 	The right hand side is uniformly bounded since $\partial_\theta \Phi_\theta(x)=\mathcal{O}(\eps)$ (thanks to the periodicity of $f_\theta$ with respect to $\theta$ and the definition of $\avg f$), $\sigma'(x)$ and $\partial_\theta F'_\theta(x)$ have at most polynomial growth, and $\sigma$ has linear growth. It follows that, after several applications of the It\^o isometry, the Cauchy-Schwarz and the Young inequalities,
 	\begin{align*}
 		\IE(|Z(t_{n+1})-Z_{n+1}|^2) \leq C h^2,
 	\end{align*}
 	and thus  $\IE(|Z(t_{n+1})-Z_{n+1}|^2 )^{\frac12}\leq Ch$.
 	
 	{Now we conclude the local strong order $1$ for $X$. First, note that 
 		$$
 		\IE(|Z(t_{n+1})-Z_{n+1}|^2 )= \IE(|\uX(t_{n+1})-\uX_{n+1}|^2) + \IE(|Y(t_{n+1})-Y_{n+1}|^2)\leq Ch^2,
 		$$
 		which implies that each component of $Z_{n+1}=(\uX_{n+1},Y_{n+1})$ converges strongly with local order $1$ to the corresponding component of $Z(t_{n+1})=(\uX(t_{n+1}),Y(t_{n+1}))$. Next, for a given $X(t_n)=X_n$, we have  
 		\begin{equation*}
 			\begin{split}
 				&\IE
 				\left(\left|X(t_{n+1})-X_{n+1}\right|^2\right) = \IE\left(\left|\Phi_{t_{n+1}/\eps}(\uX(t_{n+1}))+Y(t_{n+1})-\Phi_{t_{n+1}/\eps}(\uX_{n+1})-Y_{n+1}\right|^2\right)\\
 				&\leq \IE\left(\left|\Phi_{t_{n+1}/\eps}(\uX(t_{n+1}))-\Phi_{t_{n+1}/\eps}(\uX_{n+1})\right|^2\right)+\IE\left(\left|Y(t_{n+1})-Y_{n+1}\right|^2\right)\\
 				&+2\IE\left(\left|\left(\Phi_{t_{n+1}/\eps}(\uX(t_{n+1}))-\Phi_{t_{n+1}/\eps}(\uX_{n+1})\right)\left(Y(t_{n+1})-Y_{n+1}\right)\right|\right)\\
 				&\leq L^2\IE\left(\left|\uX(t_{n+1})-\uX_{n+1}\right|^2\right) + \IE\left(\left|Y(t_{n+1})-Y_{n+1}\right|^2\right)
 				\\
 				&+2\IE\left(\left|\Phi_{t_{n+1}/\eps}(\uX(t_{n+1}))-\Phi_{t_{n+1}/\eps}(\uX_{n+1})\right|^2\right)^{\frac12}\IE\left(\left|Y(t_{n+1})-Y_{n+1}\right|^2\right)^{\frac12}
 				\\
 				&\leq Ch^2 + Ch^2 +2(Ch)(Ch) \leq Ch^2.
 			\end{split}
 		\end{equation*}
 		Therefore,
 		$$
 		\IE
 		\left(\left|X(t_{n+1})-X_{n+1}\right|^2\right)^{\frac12}\leq Ch
 		$$
 		and the local strong order 1 is proved.}
 \end{proof}
 {
 	\begin{remark}
 		When applying the Euler-Maruyama method directly to \eqref{eq:SDE}, a term that contains $\frac1\eps \partial_\theta f_{t/\epsilon}(X(t_n))$ will appear in the remainder. This term causes the local truncation error to be of order $\mathcal{O}(\frac1\epsilon)$ which causes instability in the high frequency regime ($\epsilon\ll1$) even when Assumption \ref{eq:hypo} is satisfied. In contrast, applying Euler Maruyama method to the micro-macro system (\ref{eq:mmXbar},\ref{eq:mmY}) leads instead of $\frac1\eps \partial_\theta f_{t/\epsilon}(X(t_n))$ to the term \eqref{eq:bound} which can be bounded independently of $\epsilon$ using Assumption \ref{eq:hypo}. 
 \end{remark}}
 
 \section{Numerical experiments} \label{sec:num}

 Throughout this section (except for the last experiment) we will focus on the H\'enon-Heiles model (see \cite{BB14,HMD19}). We consider the Hamiltonian 
 \begin{equation*}
 	H(p,q)= \frac{p_1^2}{2\eps} + \frac{p_2^2}{2} + \frac{q_1^2}{2\eps} + \frac{q_2^2}{2}+q_1^2q_2 - \frac13 q_2^3.
 \end{equation*}
 Let 
 \begin{equation*}
 	\begin{cases}
 		X_1(t)=\cos\left(\frac t\eps\right)q_1(t) - \sin\left(\frac t\eps\right)p_1(t),\\
 		X_2(t)=q_2(t),\\
 		X_3(t)=\sin\left(\frac t\eps\right)q_1(t) +\cos\left(\frac t\eps\right)p_1(t),\\
 		X_4(t)=p_2(t).
 	\end{cases}
 \end{equation*}
 It can be checked that the variable $X(t)$ satisfies the following ODE \cite{CLMV20}
 \begin{equation*}
 	\frac{dX}{dt}(t)=f_{t/\eps}(X(t)),
 \end{equation*}
 with,
 \begin{equation*}
 	\begin{cases}
 		f^1_{\theta}(X)=2\sin\theta(X_1\cos\theta+X_3\sin\theta)X_2,\\
 		f^2_{\theta}(X)=X_4,\\
 		f^3_{\theta}(X)=-2\cos\theta(X_1\cos\theta+X_3\sin\theta)X_2,\\
 		f^4_{\theta}(X)=-2(X_1\cos\theta+X_3\sin\theta)^2+X_2^2-X_2.
 	\end{cases}
 \end{equation*}
 Now, we consider the SDE
 \begin{equation}\label{eq:sdenum}
 	dX=f_{t/\eps}(X)dt+\sigma(X)dW(t).
 \end{equation}
 In all our experiments $\eps=2^{-2i},~i=2,3,4,5$, final time $T=1$;
 $M$ denotes the number of computed samples and $\Delta t$ denotes the time step size.
 \subsection{Weak convergence}
 In this section we use $X^0=(0.7,0.7,0.7,0.7)$, $\Delta t=2^{-i},~i=1,\dots,5$, $M=10^4$.
 \subsubsection{Multiplicative noise}
 We consider the above SDE \eqref{eq:sdenum} with multiplicative noise where $\sigma(X)=0.2(0,0,X_1,X_2)^T$. We use the test function $\phi(X)=X_1$ to measure the weak convergence. In  Figure \ref{fig:weakmult}A we plot the weak error with respect to the time step for different values of $\eps$ (left figure) using the micro-macro method \eqref{eq:EMXbar}-\eqref{eq:EMY}. We can see that the convergence behavior looks almost the same, with weak order one,  for all the different values of $\eps$. The right picture of Figure \ref{fig:weakmult}A shows that for a fixed time step, the weak error remains almost constant when varying $\eps$. The above description applies also to the integral scheme (see Figure \ref{fig:weakmult}B).
 
 \begin{figure}[t!]
 	\centering
 	\begin{subfigure}{\linewidth}
 		\includegraphics[width=\linewidth]{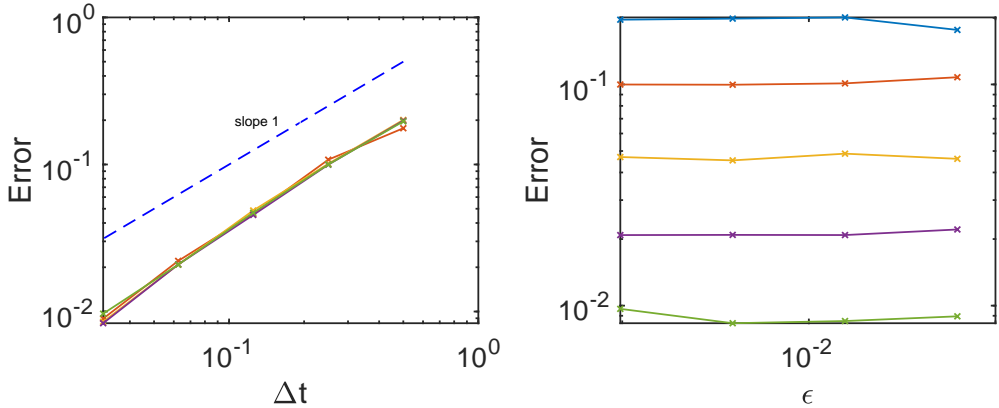}
 		\caption{Micro-Macro scheme.}
 	\end{subfigure}\\
 	\begin{subfigure}{\linewidth}
 		\includegraphics[width=\linewidth]{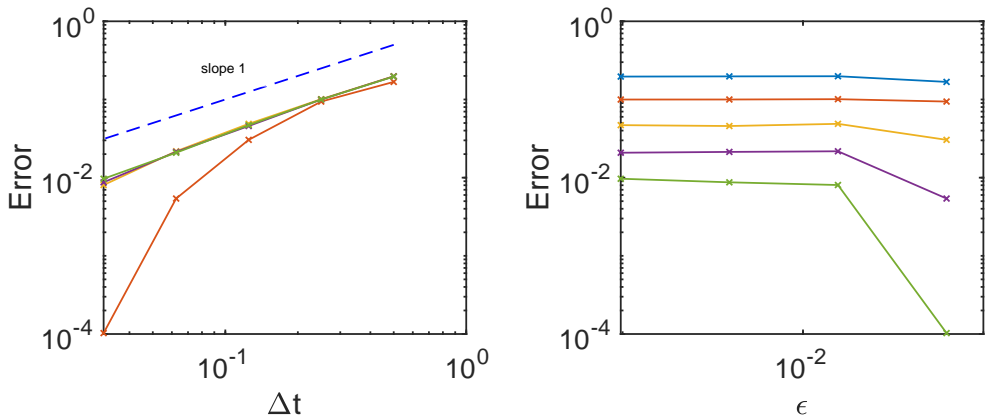}
 		\caption{Integral scheme.}
 	\end{subfigure}\\
 	\caption{Weak convergence with multiplicative noise for the micro-macro method \eqref{eq:EMXbar}-\eqref{eq:EMY} and the integral scheme\eqref{eq:int} .}
 	\label{fig:weakmult}
 \end{figure}
 
 \subsubsection{Additive noise}
 The weak error of the micro-macro scheme \eqref{eq:EMXbar}-\eqref{eq:EMY} applied to the SDE \eqref{eq:sdenum} with additive noise is shown in Figure \ref{fig:weakadd}. We set $\sigma(X)=(0,0,0.2,0.2)^T$, and we perform the test with two different test functions. We see again the uniform weak order one. 
 \begin{figure}[t!]
 	\centering
 	\begin{subfigure}{\linewidth}
 		\includegraphics[width=\linewidth]{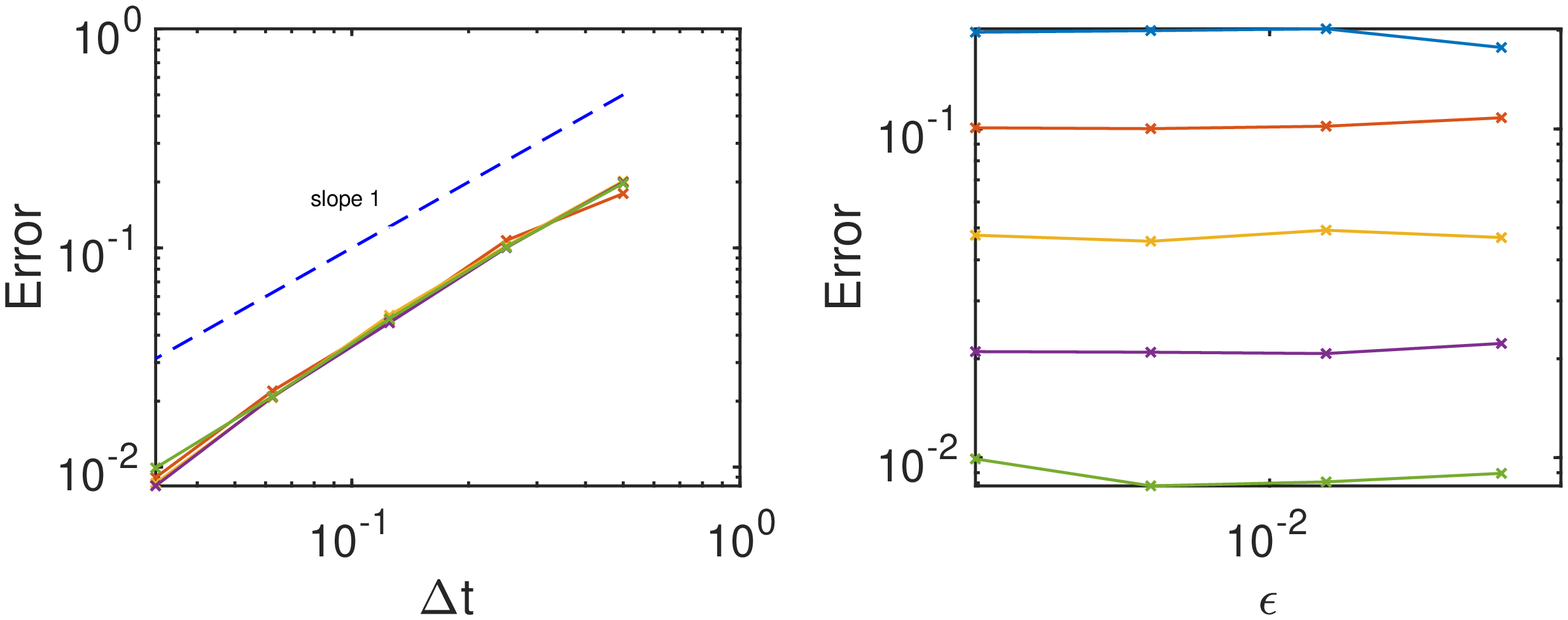}
 		\caption{$\phi(X)=X_1$.}
 	\end{subfigure}\\
 	\begin{subfigure}{\linewidth}
 		\includegraphics[width=\linewidth]{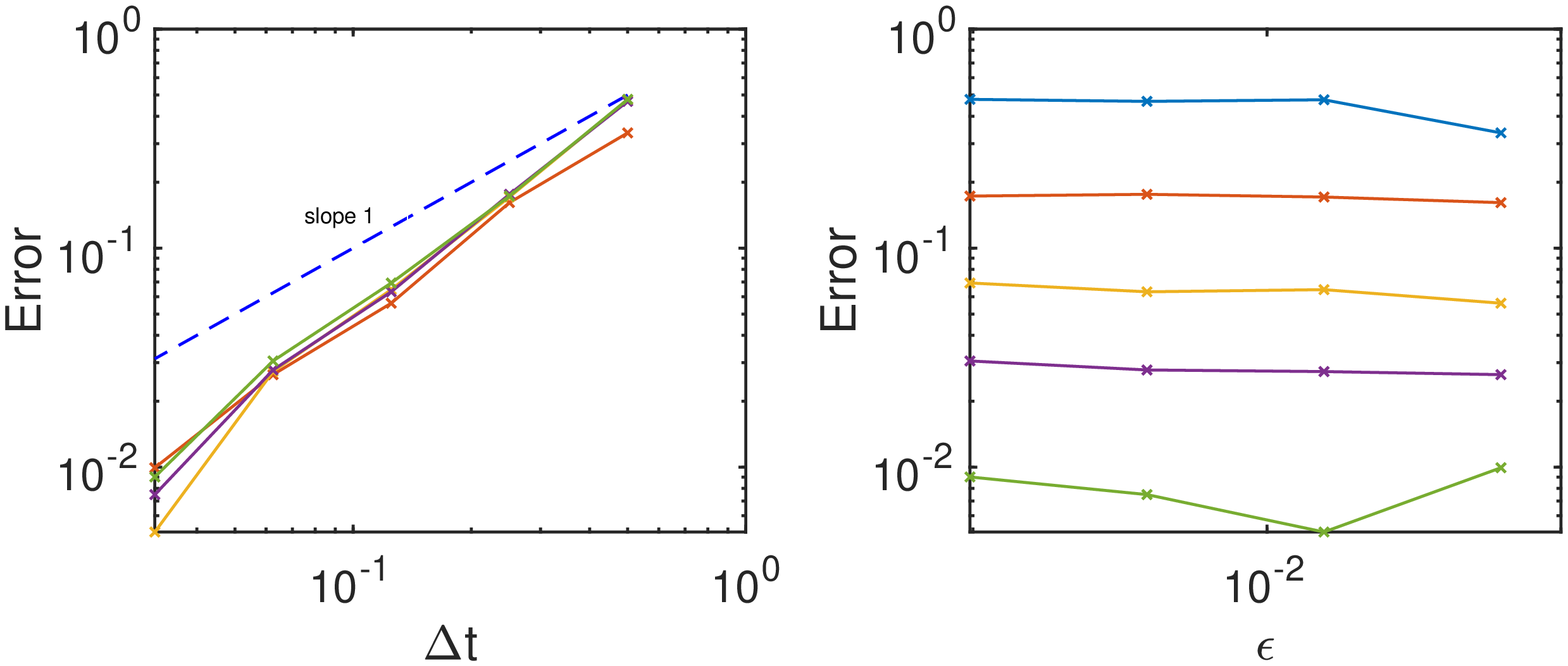}
 		\caption{$\phi(X)=\sum_{i=1}^{4}X_i$.}
 	\end{subfigure}\\
 	\caption{Weak convergence with additive noise for the micro-macro method \eqref{eq:EMXbar}-\eqref{eq:EMY}.}
 	\label{fig:weakadd}
 \end{figure}
 \subsection{Strong convergence}
 
 In this section we use $X^0=(0.12,0.12,0.12,0.12)$, $\Delta t=2^{-i},~i=4,\dots,8$, $M=10^2$. 
 
 \subsubsection{Multiplicative noise}
 We consider the above SDE \eqref{eq:sdenum} with multiplicative noise where $\sigma(X)=0.5X$. Figure \ref{fig:strongmult} shows the uniform strong order $\frac12$ for both methods.
 
 \begin{figure}[t!]
 	\centering
 	\begin{subfigure}{\linewidth}
 		\includegraphics[width=\linewidth, height=0.25\textheight]{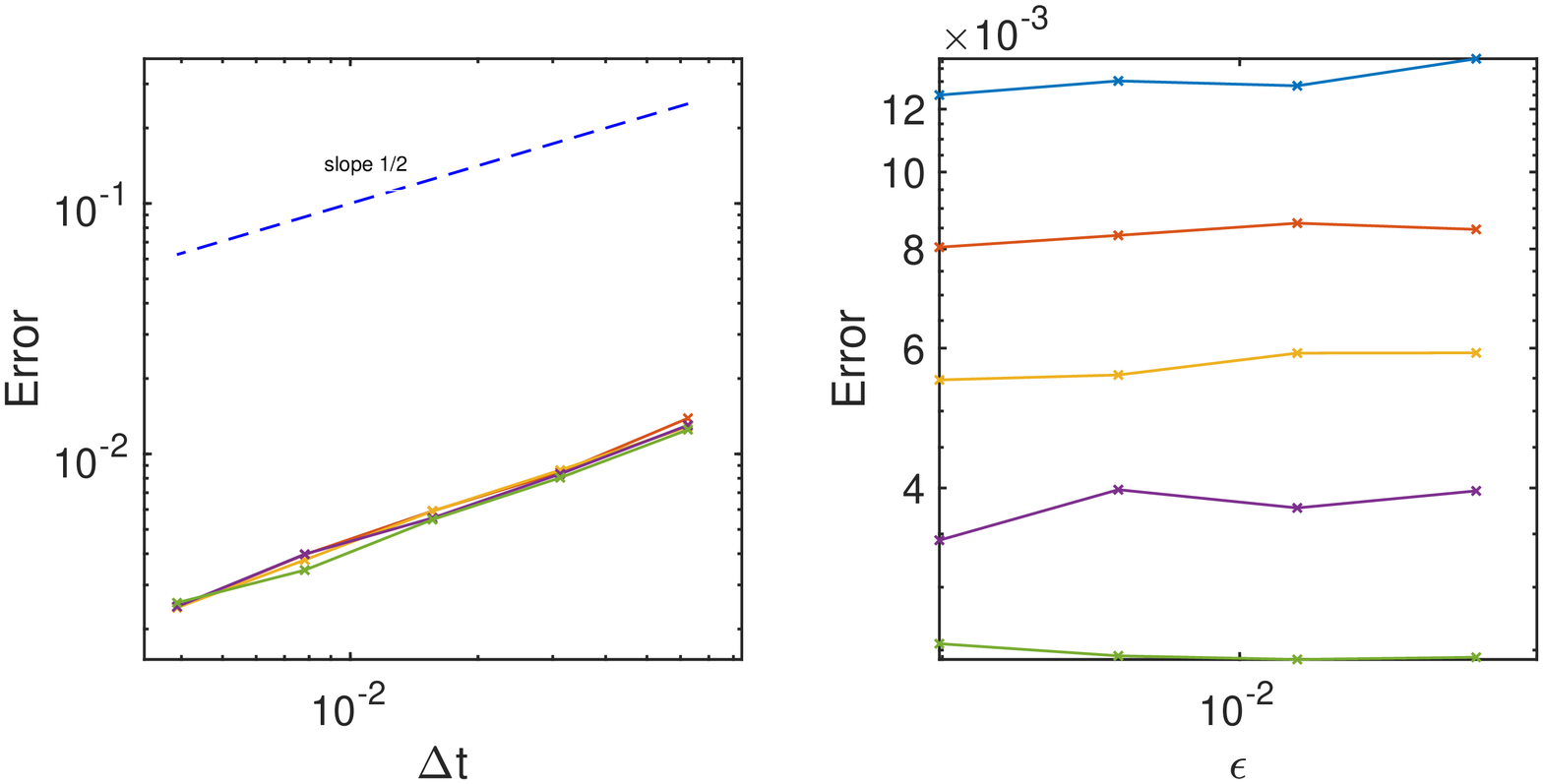}
 		\caption{Micro-Macro.}
 	\end{subfigure}\\
 	\begin{subfigure}{\linewidth}
 		\includegraphics[width=\linewidth, height=0.25\textheight]{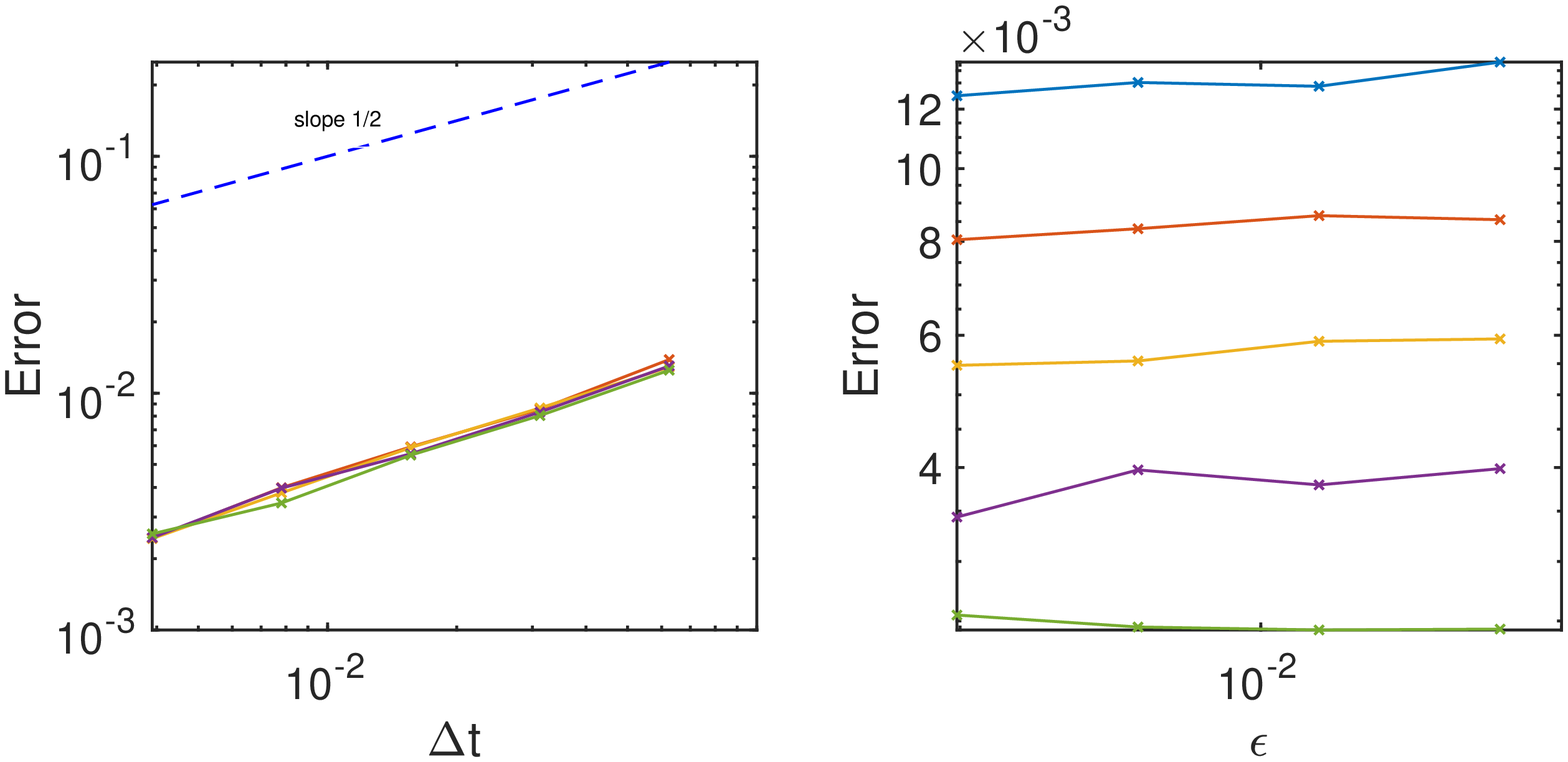}
 		\caption{Integral scheme.}
 	\end{subfigure}\\
 	\caption{Strong convergence with multiplicative noise for methods \eqref{eq:EMXbar}-\eqref{eq:EMY} and \eqref{eq:int}.}
 	\label{fig:strongmult}
 \end{figure}
 
 \subsubsection{Additive noise}
 We consider the above SDE \eqref{eq:sdenum} with additive noise where $\sigma(X)=(0,0,0.5,0.5)^T$. In addition to the uniform convergence, Figure \ref{fig:strongadd} shows strong order one for the micro-macro method \eqref{eq:EMXbar}-\eqref{eq:EMY} since when the noise is additive, Euler-Maruyama method coincides with Milstein method of strong order one. This applies to uniformly accurate methods too.
 
 \begin{figure}[t!]
 	\centering
 	\includegraphics[width=\linewidth]{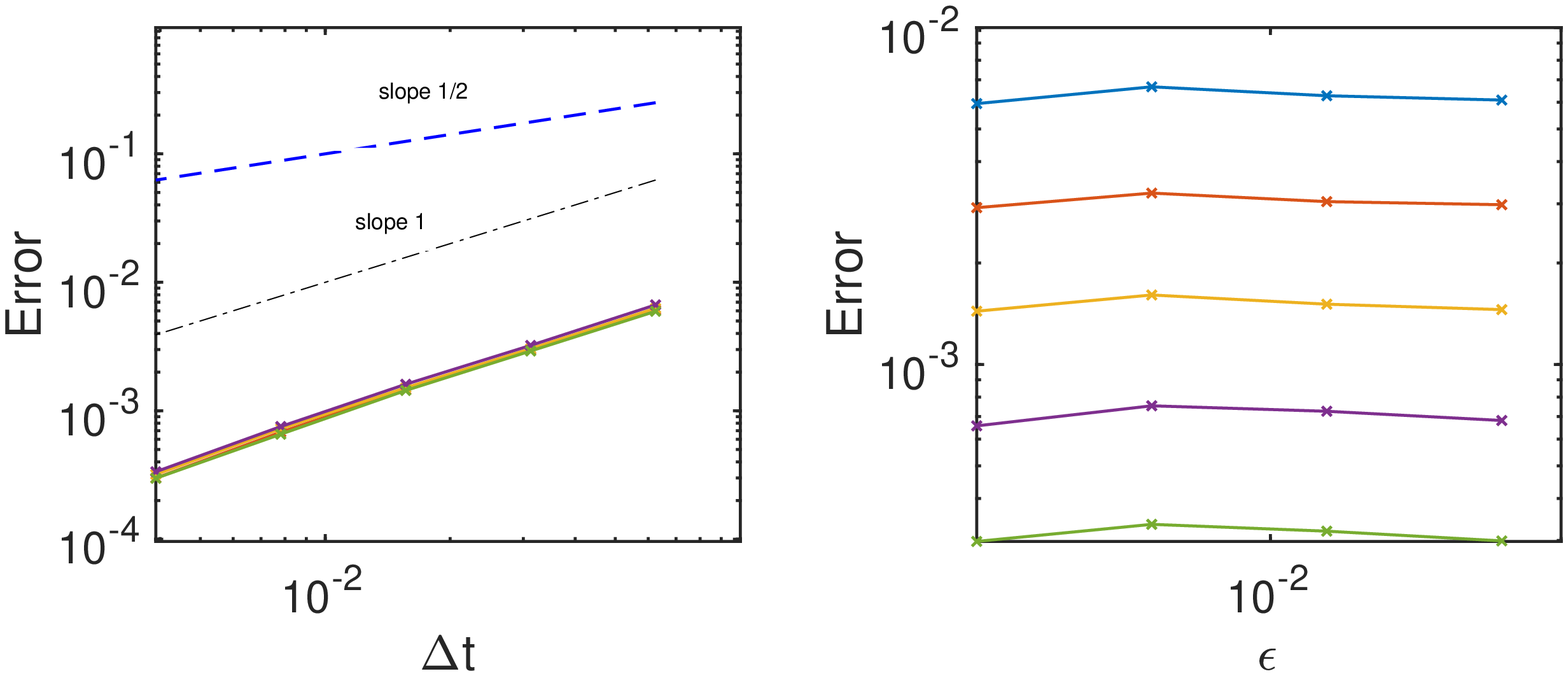}
 	\caption{Strong convergence with additive noise for the micro-macro method \eqref{eq:EMXbar}-\eqref{eq:EMY}.}
 	\label{fig:strongadd}
 \end{figure}

 \subsection{Inefficiency of Euler-Maruyama method for particular time steps}
 
 Although Euler-Maruyama method seems to work quite well, it still fails for some particular choices of time steps, while the Micro-Macro method \eqref{eq:EMXbar}-\eqref{eq:EMY} does not. We recall that the importance of uniformly accurate methods appears more when using higher order schemes. See Figure \ref{fig:unEM}.
 We consider the logistic SDE
 \begin{equation*}
 	dX=\left(X(1-X)+\sin\frac t\eps\right)dt + 0.2XdW(t),\qquad X(0)=2.
 \end{equation*}
 We set $\eps=0.1$ and we plot in Figure \ref{fig:unEM}A the reference solution (in blue) calculated with very small time step using the integral scheme and the solution obtained using EM with time step $\Delta t=0.99(2\pi\eps)$. In Figure \ref{fig:unEM}B, we plot the reference solution (in blue) calculated with very small time step using the integral scheme and the solution obtained using the uniformly accurate method \eqref{eq:EMXbar}-\eqref{eq:EMY} with time step $h=0.99(2\pi\eps)$.
 
 \begin{figure}[t!]
 	\centering
 	\begin{subfigure}{0.49\linewidth}
 		\includegraphics[width=\linewidth]{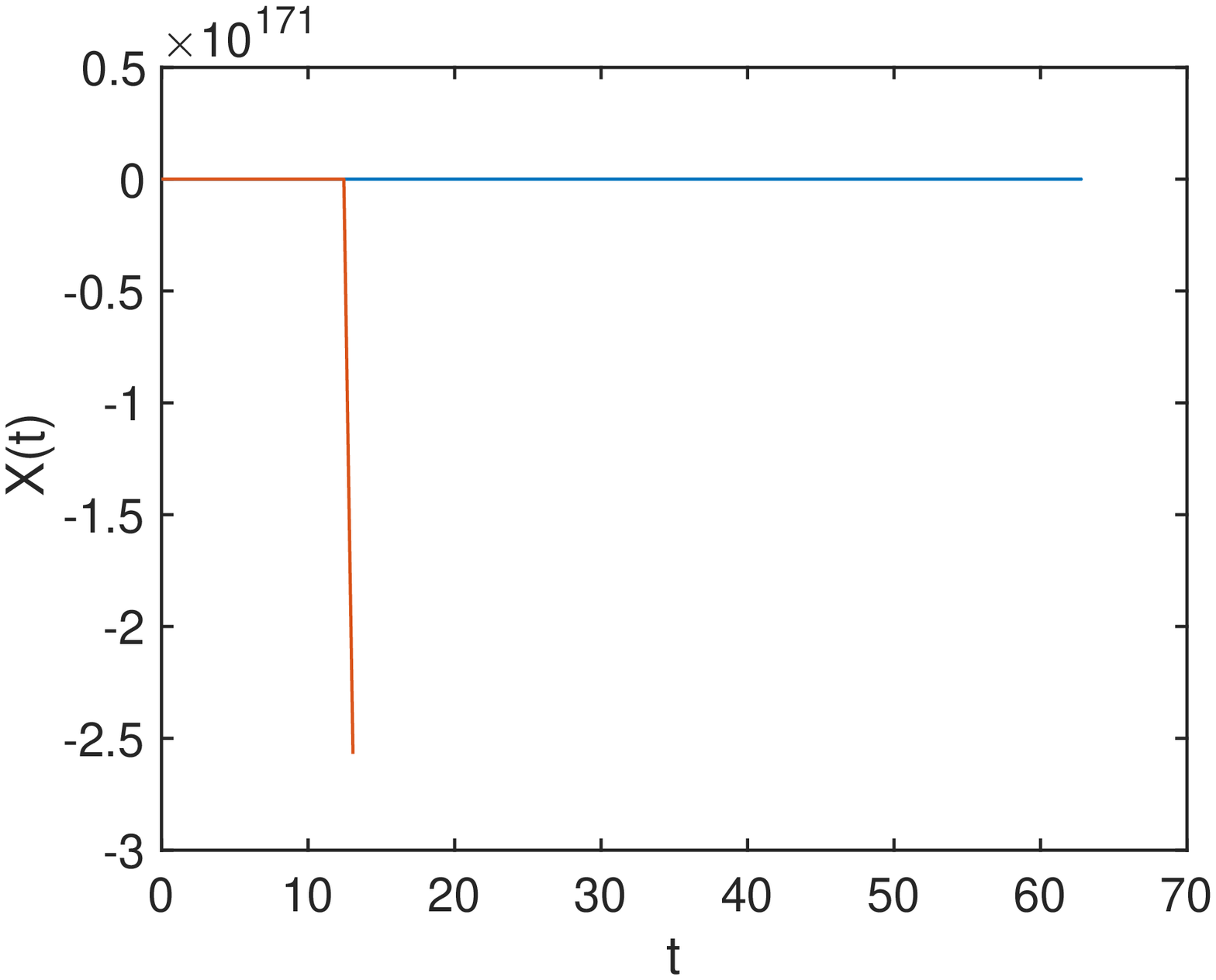}
 		\caption{Euler Maruyama}
 	\end{subfigure}
 	\begin{subfigure}{0.49\linewidth}
 		\includegraphics[width=\linewidth]{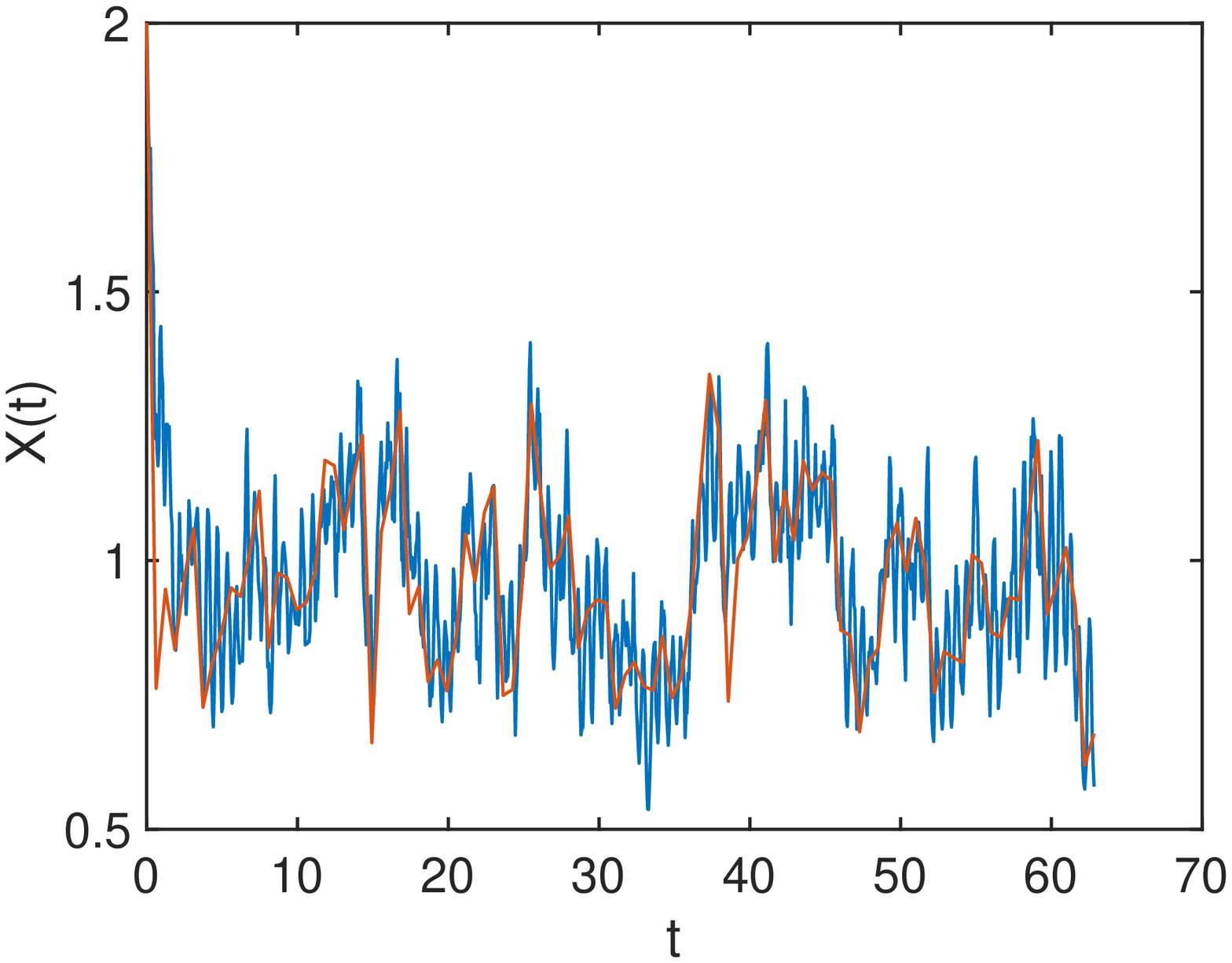}
 		\caption{Uniformly accurate method (\ref{eq:EMXbar}, \ref{eq:EMY})}
 	\end{subfigure}
 	\caption{Failure of EM for particular time steps.}
 	\label{fig:unEM}
 \end{figure}
 
 \section{Conclusion}
 
 In this work, we have introduced two uniformly accurate methods for solving numerically stochastic differential equations with oscillatory drift. The first one is the so-called integral scheme (\ref{eq:int}) and can be derived quite straightforwardly, whereas the second one is obtained through a more elaborate transformation, namely a micro-macro decomposition (\ref{eq:mmXbar}, \ref{eq:mmY}). Both schemes exhibit weak-order $1$ and strong-order $1/2$, as proved in the corresponding sections and confirmed numerically in Section \ref{sec:num}. Given their comparable performance, the first scheme is arguably better for its simplicity. However, it is our belief that the micro-macro scheme exposed here could be generalized to higher order methods which would be the stochastic counterpart of existing deterministic uniformly accurate methods \cite{BFS17,CLMV20}.
 
 \bigskip
 
 \noindent \paragraph{\textbf{Acknowledgement.}} The authors would like to thank Gilles Vilmart for useful discussions and comments.

\bibliographystyle{abbrv}
\bibliography{HLW,complete,abd_biblio,biblio_hop}

 \end{document}